\documentclass[11pt,a4,fleqn]{article}
\usepackage{graphicx}
\usepackage{amsmath,amssymb,latexsym,graphics,epsfig}
\usepackage{hyperref}
\usepackage{color}
\usepackage{amsthm}

\newtheorem{theorem}{\bf Theorem}[section]

\newtheorem{lemma}[theorem]{\bf Lemma}
\newtheorem{corollary}[theorem]{\bf Corollary}

\newtheorem{remark}[theorem]{\bf Remark}

\newtheorem{claim}{\bf Claim}

\newtheorem{conjecture}{\bf Conjecture}

\numberwithin{equation}{section}

\begin{document}
\title{{\Large Ramsey numbers of uniform loose paths and cycles}}
\author{\small  G.R. Omidi$^{\textrm{a},\textrm{b},1}$, M. Shahsiah$^{\textrm{b}}$\\
\small  $^{\textrm{a}}$Department of Mathematical Sciences,
Isfahan University
of Technology,\\ \small Isfahan, 84156-83111, Iran\\
\small  $^{\textrm{b}}$School of Mathematics, Institute for
Research
in Fundamental Sciences (IPM),\\
\small  P.O.Box: 19395-5746, Tehran,
Iran\\
\small \texttt{E-mails: romidi@cc.iut.ac.ir,
shahsiah@ipm.ir}}
\date {}
\maketitle \footnotetext[1] {\tt This research is partially
carried out in the IPM-Isfahan Branch and in part supported
by a grant from IPM (No. 92050217).} \vspace*{-0.5cm}

\begin{abstract} Recently, determining the Ramsey numbers of loose paths and cycles in uniform hypergraphs
has received considerable attention. It has been shown that  the $2$-color Ramsey number
of a $k$-uniform loose cycle
$\mathcal{C}^k_n$, $R(\mathcal{C}^k_n,\mathcal{C}^k_n)$, is asymptotically $\frac{1}{2}(2k-1)n$. Here we
conjecture that for any $n\geq m\geq 3$ and $k\geq 3,$
$$R(\mathcal{P}^k_n,\mathcal{P}^k_m)=R(\mathcal{P}^k_n,\mathcal{C}^k_m)=R(\mathcal{C}^k_n,\mathcal{C}^k_m)+1=(k-1)n+\lfloor\frac{m+1}{2}\rfloor.$$
Recently the case $k=3$ is proved by the
authors. In this paper, first we show that this
conjecture is true for $k=3$ with a much shorter proof.
Then, we show that for fixed $m\geq 3$ and  $k\geq 4$ the conjecture is equivalent to (only) the last equality for any  $2m\geq n\geq m\geq 3$.
Consequently, the proof for $m=3$ follows.


\end{abstract}

\section{\normalsize Introduction}

For given $k$-uniform hypergraphs $\mathcal{G}$ and $\mathcal{H},$
the \textit{Ramsey number} $R(\mathcal{G},\mathcal{H})$ is
defined to be
the smallest integer $N$ so that in every red-blue coloring of the
edges of the complete $k$-uniform hypergraph $\mathcal{K}^k_N$ there is  a red copy of $\mathcal{G}$ or a blue copy of $\mathcal{H}.$
There are various definitions  for paths and cycles in
hypergraphs. The case we focus on here is called {\it loose}. A {\it
$k$-uniform  loose cycle} $\mathcal{C}_n^k$ (shortly, a {\it cycle
of length $n$})  is a hypergraph with vertex set
$\{v_1,v_2,\ldots,v_{n(k-1)}\}$ and the set of $n$ edges
$e_i=\{v_1,v_2,\ldots, v_k\}+(i-1)(k-1)$, $i=1,2,\ldots, n$. Here,
we use mod $n(k-1)$ arithmetic and adding a number $t$ to a set
$H=\{v_1,v_2,\ldots, v_k\}$ means a shift, i.e. the set obtained
by adding $t$ to subscripts of each element of $H$. Similarly, a
{\it $k$-uniform  loose path} $\mathcal{P}_n^k$ (shortly, a {\it
path of length $n$}) is a hypergraph with vertex set
$\{v_1,v_2,\ldots,v_{n(k-1)+1}\}$  and the set of $n$ edges
$e_i=\{v_1,v_2,\ldots, v_k\}+(i-1)(k-1)$, $i=1,2,\ldots, n$.
For
an edge $e_i=\{v_{(i-1)(k-1)+1},v_{(i-1)(k-1)+2},\ldots, v_{i(k-1)+1}\}$ of a given loose path (also a given loose cycle)
$\mathcal{K}$, the first vertex, $v_{(i-1)(k-1)+1},$ and the last vertex, $v_{i(k-1)+1},$ are denoted by
$f_{\mathcal{K},e_i}$ and
$l_{\mathcal{K},e_i}$, respectively.\\

%

The problem  of determining or estimating Ramsey numbers is one of
the most important problems in combinatorics which has been of
interest to many investigators. In contrast to the graph case, there are only a few results on the Ramsey numbers of hypergraphs. Recently, this topic
has received considerable attention. The investigation of the
Ramsey numbers of hypergraph loose paths and cycles was initiated
by  Haxell et al. (see \cite{HLPRRSS}).
Indeed, they determined the asymptotic value of the Ramsey number
of $3$-uniform loose cycles. This result
was extended by Gy\'{a}rf\'{a}s, S\'{a}rk\"{o}zy and Szemer\'{e}di
\cite[Theorem 2]{GSS} to $k$-uniform
loose cycles as follows.

\begin{theorem}{\rm \cite{GSS}}\label{Gyarfas} For all  $\eta>0$ there exists $n_0=n_0(\eta)$ such that for every $n> n_0,$
every 2-coloring of $\mathcal{K}^{k}_N$ with
$N=(1+\eta)\frac{1}{2}(2k-1)n$ contains a monochromatic copy of
$\mathcal{C}^k_n.$
\end{theorem}

Some interesting results were obtained on the exact
values of the Ramsey numbers of loose paths and cycles.
Gy\'{a}rf\'{a}s and Raeisi \cite{GR} determined the values of the
Ramsey numbers of two $k$-uniform loose triangles and two $k$-uniform quadrangles. In \cite{MORS}, the authors proved that for every $n\geq
\lfloor\frac{5m}{4}\rfloor$,
$R(\mathcal{P}^3_n,\mathcal{P}^3_m)=2n+\lfloor\frac{m+1}{2}\rfloor.$
Recently, the Ramsey
numbers of 3-uniform loose paths and loose cycles are completely determined; see \cite{OS}. These results motivate us to pose the following conjecture:

\begin{conjecture}\label{path-cycle}
Let $k\geq 3$ be an integer number. For any $n\geq m \geq 3$,
\begin{equation}\label{eq(2)}
R(\mathcal{P}^k_n,\mathcal{P}^k_m)=R(\mathcal{P}^k_n,\mathcal{C}^k_m)=R(\mathcal{C}^k_n,\mathcal{C}^k_m)+1=(k-1)n+\lfloor\frac{m+1}{2}\rfloor.
\end{equation}
\end{conjecture}
In the next section, we provide a proof of Conjecture \ref{path-cycle} when $k=3$ (a much shorter proof to that of \cite{OS}). For this purpose, first we show that Conjecture \ref{path-cycle} is equivalent to the following conjecture (see Theorem \ref{equivalent1}).
\begin{conjecture}\label{cycle}
Let $k\geq 3$ be an integer number. For every $n\geq m \geq 3$,
\begin{eqnarray*}\label{3}
\hspace{3
cm}R(\mathcal{C}^k_n,\mathcal{C}^k_m)=(k-1)n+\lfloor\frac{m-1}{2}\rfloor.
\end{eqnarray*}
\end{conjecture}
\noindent Then we will  prove Conjecture \ref{cycle}  for $k=3.$
In Section $3$, we will demonstrate that for fixed $m\geq 3$ and  $k\geq 4$  Conjecture \ref{path-cycle} is equivalent to (only) the last equality for any  $2m\geq n\geq m\geq 3$. More precisely, we will
show that for fixed $m\geq 3$ and  $k\geq 4,$ Conjecture
\ref{cycle} is true for each $n\geq m$ if and only if it is true   for
 each  $2m\geq n\geq m\geq 3$. So using Theorem \ref{equivalent1} we are done.
Subsequently, in the last section, we conclude that
Conjecture \ref{path-cycle}  holds for $m=3$ and every $n\geq 3.$

The following
lemma \cite[Lemma 1]{GR} shows that the values of the Ramsey numbers in Conjecture \ref{path-cycle} are lower bounds
for the claimed Ramsey numbers.

\begin{lemma}{\rm \cite{GR}}\label{lower bound}
For every $n\geq m \geq 2$ and $k\geq 3,$
$(k-1)n+\Big\lfloor\frac{m+1}{2}\Big\rfloor$   is a lower bound
for  both $R(\mathcal{P}^k_n,\mathcal{P}^k_m)$ and
$R(\mathcal{P}^k_n,\mathcal{C}^k_m).$ Moreover,
$R(\mathcal{C}^k_n,\mathcal{C}^k_m)\geq
(k-1)n+\Big\lfloor\frac{m-1}{2}\Big\rfloor$.
\end{lemma}
\noindent Therefore, in  this paper in order
to determine the Ramsey numbers, it suffices to verify that the
known lower bounds are also upper bounds.\\



 Throughout the paper, for
a 2-edge colored hypergraph $\mathcal{H}$ we denote by
$\mathcal{H}_{\rm red}$ and $\mathcal{H}_{\rm blue}$ the induced
hypergraphs on red edges and blue edges, respectively. Also, the number of
vertices and edges of $\mathcal{H}$ are
denoted by $|\mathcal{H}|$ and $\|\mathcal{H}\|$.

\section{ $3$-uniform loose paths and cycles}

In this section, we present a proof of Conjecture \ref{path-cycle} when $k=3$ (an alternative proof to that of \cite{OS}). First, we show that this conjecture is equivalent to Conjecture \ref{cycle}. For this purpose, we sketch how the last equality of (\ref{eq(2)}) for given $n\geq m\geq 2$, leads to
determine the values of $R(\mathcal{P}^k_n,\mathcal{C}^k_m)$ and $R(\mathcal{P}^k_n,\mathcal{P}^k_{m-1})$ and also $R(\mathcal{P}^k_n,\mathcal{P}^k_n)$ when $n=m$.

The fact that $(k-1)n+\lfloor\frac{m+1}{2}\rfloor$ is a
lower bound for $R(\mathcal{P}^k_{n},\mathcal{C}^k_{m})$ follows from Lemma \ref{lower bound}. To see that it is the upper bound, assume that
$\mathcal{K}^k_{(k-1)n+\lfloor\frac{m+1}{2}\rfloor}$ is $2$-edge
colored red and blue. Since
$$R(\mathcal{C}^k_{n},\mathcal{C}^k_{m})=(k-1)n+\lfloor\frac{m-1}{2}\rfloor<
(k-1)n+\lfloor\frac{m+1}{2}\rfloor,$$ we have  a red copy of
$\mathcal{C}^k_{n}$ or a blue copy of $\mathcal{C}^k_{m}$. If
there is a blue copy of $\mathcal{C}^k_{m}$, we are done. Otherwise, the
existence of a red copy of $\mathcal{C}^k_{n}$ implies that there
is a red copy of $\mathcal{P}^k_{n}$ by \cite[Lemma 2]{GR}. Now we show that $$R(\mathcal{P}^k_{n},\mathcal{P}^k_{m-1})=(k-1)n+\lfloor\frac{m}{2}\rfloor.$$ To see this, assume that
$\mathcal{K}^k_{(k-1)n+\lfloor\frac{m}{2}\rfloor}$ is $2$-edge
colored red and blue. Again, we have  a red copy of
$\mathcal{C}^k_n$ or a blue copy of $\mathcal{C}^k_{m}.$ If the
first case holds, by \cite[Lemma 2]{GR}, we do not have
any thing to prove. Otherwise, a blue copy of  $\mathcal{C}^k_{m}$
contains a blue copy of $\mathcal{P}^k_{m-1}.$  This observation and Lemma
\ref{lower bound} complete the proof.

Now let $n=m$. By
applying Lemma 2 of \cite{GR} the existence of a
monochromatic $\mathcal{C}^k_{n}$ in a $2$-edge colored
$\mathcal{K}^k_{(k-1)n+\lfloor\frac{n+1}{2}\rfloor}$ implies that
there is a monochromatic $\mathcal{P}^k_{n}$. So using Lemma
\ref{lower bound} we have
$$R(\mathcal{P}^k_{n},\mathcal{P}^k_{n})=(k-1)n+\lfloor\frac{n+1}{2}\rfloor.$$
In fact, we have the
following theorem.

\begin{theorem}\label{connection}
Let $n\geq m\geq 2$ be given integers and $R(\mathcal{C}^k_{n},\mathcal{C}^k_{m})=(k-1)n+\lfloor\frac{m-1}{2}\rfloor.$ Then
$R(\mathcal{P}^k_n,\mathcal{C}^k_m)=(k-1)n+\lfloor\frac{m+1}{2}\rfloor$ and $R(\mathcal{P}^k_n,\mathcal{P}^k_{m-1})=(k-1)n+\lfloor\frac{m}{2}\rfloor$. Moreover, for $n=m$ we have $R(\mathcal{P}^k_n,\mathcal{P}^k_m)=(k-1)n+\lfloor\frac{m+1}{2}\rfloor$.
\end{theorem}

Using Theorem \ref{connection} we have the following result.
\begin{theorem}\label{equivalent1}
Two Conjectures {\rm\ref{path-cycle}} and {\rm \ref{cycle}} are equivalent.
\end{theorem}


In the rest of this section, we will demonstrate that Conjecture \ref{cycle} is true for $k=3.$ For this purpose, we need some definitions.\\

Let $\mathcal{H}$ be a 2-edge colored complete
3-uniform hypergraph, $\mathcal{P}$ be a loose path in $\mathcal{H}$
and $W$ be a set of vertices with $W\cap
V(\mathcal{P})=\emptyset$. By a {\it $\varpi_S$-configuration}, we
mean a copy of $\mathcal{P}^3_2$ with edges $\{x,a_1,a_2\}$ and
$\{a_2,a_3,y\}$ so that $\{x,y\}\subseteq W$  and
 $S=\{a_{j} : 1\leq j
\leq 3\}\subseteq (e_{i-1}\setminus \{f_{\mathcal{P},e_{i-1}}\})\cup e_{i}
\cup e_{i+1}$ is a set of unordered vertices of
 three  consecutive edges of
$\mathcal{P}$ with  $|S\cap (e_{i-1}\setminus \{f_{\mathcal{P},e_{i-1}}\})|\leq 1.$
  The vertices $x$ and $y$ are called {\it the end vertices} of
this configuration. A $\varpi_{S}$-configuration,
$S\subseteq (e_{i-1}\setminus \{f_{\mathcal{P},e_{i-1}}\})\cup e_{i}
\cup e_{i+1}$, is {\it good} if at least one of the vertices of
$e_{i+1}\setminus e_{i}$ is not in $S$.
We say that a
monochromatic path $\mathcal{P}=e_1e_2\ldots e_n$ is {\it maximal
with respect to} $W\subseteq
V(\mathcal{H})\setminus V(\mathcal{P})$ (in brief,
{\it maximal w.r.t.} $W$) if there is no $W'\subseteq W$ so
that for some $1\leq r\leq n$ and $1\leq i\leq n-r+1,$
\begin{eqnarray*}\mathcal{P}'=e_1e_2\ldots e_{i-1}e'_ie'_{i+1}\ldots e'_{i+r}e_{i+r}\ldots e_n,\end{eqnarray*} is a monochromatic path with $n+1$ edges and
 the following properties:

 \begin{itemize}
\item[(i)] $V(\mathcal{P}')=V(\mathcal{P})\cup W'$,
 \item[(ii)] if $i=1$, then
$f_{\mathcal{P}',e'_i}=f_{\mathcal{P},e_i}$,
\item[(iii)] if
 $i+r-1=n$, then
$l_{\mathcal{P}',e'_{i+r}}=l_{\mathcal{P},e_n}$.
\end{itemize}
 Clearly, if $\mathcal{P}$ is maximal w.r.t. $W$, then it is maximal w.r.t. every
$W'\subseteq W$ and also every loose path $\mathcal{P}'$ which is a sub-hypergraph of $\mathcal{P}$ is again maximal w.r.t. $W$.

\begin{lemma}\label{spacial configuration k=3 2}
Let  $\mathcal{H}=\mathcal{K}^3_{\ell}$ be $2$-edge colored
red and blue and   let $\mathcal{P}=e_1e_2\ldots e_n$ \ $\subseteq
\mathcal{H}_{\rm red}$ be a maximal path w.r.t. $W,$ where  $W\subseteq V(\mathcal{H})\setminus V(\mathcal{P})$ and $|W|\geq 3$.
Set $A_1=\{f_{\mathcal{P},e_{1}}\}$ and
$A_i=e_{i-1}\setminus\{f_{\mathcal{P},e_{i-1}}\}$ for $i>1$.
 Then for every two  consecutive edges $e_i$ and $e_{i+1}$
of $\mathcal{P}$ and for each $u\in A_i$ there is a good
$\varpi_S$-configuration, say $C$, in $\mathcal{H}_{\rm blue}$
with end vertices in $W$ and
\begin{eqnarray*}
S\subseteq ((e_i\setminus
\{f_{\mathcal{P},e_{i}}\})\cup \{u\})  \cup (e_{i+1}\setminus \{v\}),
\end{eqnarray*}
 for some
$v\in A_{i+2}$ such that each vertex of $W,$ with the exception of
at most one vertex, can be considered as an end vertex of $C$.
\end{lemma}
\begin{proof}{ Let $\mathcal{P}=e_1e_2\ldots e_n$ \ $\subseteq
\mathcal{H}_{\rm red}$ be a maximal path w.r.t.   $W\subseteq V(\mathcal{H})\setminus V(\mathcal{P}),$ where
$$e_i=\{v_{2i-1},v_{2i},v_{2i+1}\}, \hspace{1 cm} i=1,2,\ldots, n.$$
Assume that
 $W=\{x_1,\ldots,x_t\}\subseteq V(\mathcal{H})\setminus V(\mathcal{P}).$ Consider the edges $e_i=\{v_{2i-1},v_{2i},v_{2i+1}\}$ and $e_{i+1}=\{v_{2i+1},v_{2i+2},v_{2i+3}\}.$ If the edge  $\{u,v_{2i},x\}$ (resp.  the edge $\{v_{2i+2},v_{2i+3},x\}$) is
red  for some $x\in W$, then the maximality  of $\mathcal{P}$
w.r.t. $W$ implies that for arbitrary vertices
 $x^{\prime}\neq x^{\prime\prime}\in W\setminus \{x\}$ the edges $\{x^{\prime},v_{2i+1},v_{2i}\}$ and
 $\{v_{2i},v_{2i+2},x^{\prime\prime}\}$ (resp. $\{x^{\prime},v_{2i+1},v_{2i+2}\}$ and $\{v_{2i+2},v_{2i},x^{\prime\prime}\}$)
 are blue and there is a good $\varpi_S$-configuration $C=\{x^{\prime},v_{2i+1},v_{2i}\}\{v_{2i},v_{2i+2},x^{\prime\prime}\}$
 (resp. $C=\{x^{\prime},v_{2i+1},v_{2i+2}\}$
 $\{v_{2i+2},v_{2i},x^{\prime\prime}\}$) with $$S=\{v_{2i},v_{2i+1},v_{2i+2}\}\subseteq ((e_i\setminus
\{v_{2i-1}\})\cup \{u\})  \cup (e_{i+1}\setminus \{v_{2i+3}\}).$$
 So we may assume that for each vertex  $x\in W$ both
  edges $\{u,v_{2i},x\}$ and $\{v_{2i+2},v_{2i+3},x\}$ are blue.
  If there is a vertex $y\in W$ such that at least one of the edges $f_1=\{u,v_{2i+1},y\}$,
  $f_2=\{v_{2i},v_{2i+1},y\}$,
  $f_3=\{u,v_{2i+2},y\}$ or $f_4=\{v_{2i},v_{2i+2},y\}$, say $f$,
  is blue, then there is a good $\varpi_S$-configuration $C=\{u,v_{2i},x\}f$, where $x\neq y$, with
   $$S=\{u,v_{2i}\}\cup (f\setminus\{y\})\subseteq ((e_i\setminus
\{v_{2i-1}\})\cup \{u\})  \cup (e_{i+1}\setminus \{v_{2i+3}\}).$$
   Otherwise, we may assume that for every $y\in
  W$ the edges  $f_1, f_2, f_3$ and $f_4$  are red.
 Therefore,
  maximality of $\mathcal{P}$ w.r.t. $W$ implies that for every $y'\in W$ the
  edge $\{v_{2i},v_{2i+3},y'\}$ is blue (otherwise,
replacing $e_ie_{i+1}$ by $f_3f_2\{v_{2i},v_{2i+3},y'\}$, where
$y\neq y'$, in $\mathcal{P}$ yields a red path $\mathcal{P'}$
greater than $\mathcal{P}$; this is a contradiction). Thus, for
every $a\neq b\in W$,
$C=\{u,a,v_{2i}\}\{v_{2i},v_{2i+3},b\}$ is a good
$\varpi_{S}$-configuration with the desired properties, so that
$$S=\{u,v_{2i},v_{2i+3}\}\subseteq ((e_i\setminus
\{v_{2i-1}\})\cup \{u\})  \cup (e_{i+1}\setminus \{v_{2i+2}\}).$$}\end{proof}

\begin{corollary}\label{there is a Pl k=3}
Let $\mathcal{H}=\mathcal{K}_{\ell}^3$ be $2$-edge colored red and blue
and $\mathcal{P}=e_1e_2\ldots e_n$ with $n\geq 2$ be a maximal red
path w.r.t. $W,$ where $W\subseteq V(\mathcal{H})\setminus
V(\mathcal{P})$ and $|W|\geq 3$. Then for some $r\geq 0$ and
$W'\subseteq W$ there is a blue path $\mathcal{Q}=f_1f_2\ldots
f_q$ between $W'$ and $\overline{\mathcal{P}}=e_1e_2\ldots
e_{n-r}$ so that
$W'=\{f_{\mathcal{Q},f_1}\}\cup\{l_{\mathcal{Q},f_{2i}}|1\leq
i\leq q/2\}$ and $V(\mathcal{Q})\setminus W'\subseteq
\overline{\mathcal{P}}$. Moreover, $\mathcal{Q}$ does not contain
at least one of the vertices of $e_{n-r}\setminus e_{n-r-1}$
 as a vertex,
$\|\mathcal{Q}\|=q=2(|W'|-1)=n-r$ and either $x=|W\setminus
W'|\in\{0,1\}$ or $x\geq 2$ and $0\leq r \leq 1$.
\end{corollary}

\begin{proof}{Let $\mathcal{P}=e_1e_2\ldots e_{n}$ be a maximal red path
w.r.t. $W\subseteq V(\mathcal{H})\setminus V(\mathcal{P})$ where
$$e_i=\{v_{2i-1},v_{2i},v_{2i+1}\}, \hspace{1 cm} i=1,2,\ldots,
n.$$

\noindent{\bf Step 1:} Set  $\mathcal{P}_1=\mathcal{P}$, $W_1=W$ and
$\mathcal{P}'_1=e_{1}e_{2}$. Since $\mathcal{P}$ is maximal w.r.t.
$W_1$, using Lemma
\ref{spacial configuration k=3 2} there is a good $\varpi_S$-configuration, say $\mathcal{Q}_1,$  with end vertices $x_1$ and $y_1$ in $W_1$ so that $S\subseteq \mathcal{P}'_1$ and
 $\mathcal{Q}_1$ does not contain at least one of the
vertices of $e_2\setminus e_1,$ say $u_1$. Set $W_2=W$ and $\mathcal{P}_2=e_3e_4\ldots e_n.$ If  $|W_2|\leq 3$ or $\|\mathcal{P}_2\|\leq 1$,
 then $\mathcal{Q}=\mathcal{Q}_1$ is a blue path between $W'=W \cap V(\mathcal{Q}_1)$ and $\overline{\mathcal{P}}=e_1e_2$
 with the desired properties.
Otherwise, go to Step 2.\\\\
 {\bf Step k ($k\geq 2$):} Clearly $|W_k|>3$ and $\|\mathcal{P}_k\|> 1.$
Set $\mathcal{P}'_k=((e_{2k-1}\setminus\{f_{\mathcal{P},e_{2k-1}}\})\cup \{u_{k-1}\})e_{2k}$. Since $\mathcal{P}$
 is maximal w.r.t. $W_k$,
 using Lemma \ref{spacial configuration k=3 2} there is a a good $\varpi_S$-configuration, say $\mathcal{Q}_k,$ with end vertices  in $W_k$ so that $S\subseteq \mathcal{P}'_k$ and
 $\mathcal{Q}_k$ does not contain at least one of the
vertices of $e_{2k}\setminus e_{2k-1},$ say $u_{k}.$ Since each vertex of $W_k,$ with the exception at most one, can be considered as an end vertex of $\mathcal{Q}_k,$ we may assume that $\bigcup_{i=1}^{k}\mathcal{Q}_i$ is a blue path with end
 vertices $x_k$ and $y_k$ in $W_k$. Set $\mathcal{P}_{k+1}=e_{2k+1}e_{2k+2}\ldots e_n$ and
 $W_{k+1}=(W\setminus \bigcup_{i=1}^{k}V(\mathcal{Q}_i))\cup\{x_{k},y_{k}\}.$ If $|W_{k+1}|\leq 3$ or $\|\mathcal{P}_{k+1}\|\leq 1$,
 then $\mathcal{Q}=\bigcup_{i=1}^{k}\mathcal{Q}_i$ is a blue path between $W'=W \cap (\bigcup_{i=1}^{k} V(\mathcal{Q}_i))$ and $\overline{\mathcal{P}}=e_1e_2\ldots e_{2k}$
 with the desired properties. Otherwise, go to Step $k+1.$\\

Let $t\geq 2$ be the minimum integer for which we have  $|W_t|\leq
3$ or $\|\mathcal{P}_t\|\leq 1$. Let $W'=W \cap (\bigcup_{i=1}^{t-1} V(\mathcal{Q}_i))$. Clearly
$|W\setminus W'|=0,1$ or $|W\setminus W'|\geq
 2$ and $0\leq \|\mathcal{P}_t\|\leq 1$. So $\mathcal{Q}=\bigcup_{i=1}^{t-1}\mathcal{Q}_i$ is a blue path between $\overline{\mathcal{P}}=e_1e_2\ldots
 e_{n-r}$ and $W'$ with the desired properties, where $r=\|\mathcal{P}_t\|$. Note that, we have  $|W'|=t$ and $n-r=2(t-1).$}\end{proof}

\begin{lemma}\label{cn-1 implies cm k=3}
Let $n\geq m\geq 3,$ $(n,m)\neq (3,3),(4,3),(4,4)$  and
$\mathcal{H}=\mathcal{K}^3_{2n+\lfloor\frac{m-1}{2}\rfloor}$ be
$2$-edge colored red and blue. Assume that there is no copy of
$\mathcal{C}^3_{n}$ in $\mathcal{H}_{\rm red}$ and
$\mathcal{C}=\mathcal{C}^3_{n-1}$ is a loose cycle in
$\mathcal{H}_{\rm red}$. Then there is a copy of
$\mathcal{C}^3_{m}$ in $\mathcal{H}_{\rm blue}$.
\end{lemma}

\begin{proof}{Let
$\mathcal{C}=e_1e_2\ldots e_{n-1}$ be a copy of
$\mathcal{C}_{n-1}^3$ in $\mathcal{H}_{\rm red}$ with  edges
$e_i=\{v_1,v_2,v_3\}+2(i-1)$ (mod $2(n-1)$), $i=1,\ldots, n-1$ and
 $W=V(\mathcal{H})\setminus V(\mathcal{C})$. We consider the
following cases.\\

\noindent \textbf{Case 1. } For some edge $e_i=\{v_{2i-1},v_{2i},v_{2i+1}\},$  $1\leq i\leq n-1,$ there is
 a vertex $z\in W$ such that at least one of the edges
 $\{z,v_{2i},v_{2i+1}\}$ or $\{v_{2i-1},v_{2i},z\}$ is red.\\

\noindent Assume that the edge $g=\{z,v_{2i},v_{2i+1}\}$ is red. Set  $\mathcal{P}=e_{i+1}e_{i+2}\ldots e_{n-1}e_1e_2\ldots
e_{i-2}e_{i-1}$ and $W_0=W\setminus \{z\}$ (If the edge  $\{v_{2i-1},v_{2i},z\}$ is red,  consider $ \mathcal{P}=e_{i-1}e_{i-2}\ldots
e_2e_1e_{n-1}e_{n-2}$ $\ldots e_{i+2}e_{i+1}$ and do the following process to get a blue copy of $\mathcal{C}^3_m$).

First let $m\leq 4$. Therefore, $|W_0|=2.$ Assume that $W_0=\{u,v\}.$ We
show that $\mathcal{H}_{\rm blue}$ contains $\mathcal{C}^3_m$ for
each $m\in \{3,4\}$. Since $n\geq 5$ and there is no red copy of
$\mathcal{C}_n^3,$  the edges $f_1=\{u,v_{2i-2},v_{2i}\}, f_2=\{v_{2i},v,v_{2i-1}\}, f_3=\{v_{2i-1},z,u\}$
are blue (if the edge $f_j$ for $1\leq j \leq 3$ is red, then $f_jge_{i+1}\ldots e_{n-1}e_1\ldots e_{i-1}$ is a red copy of $\mathcal{C}_n^3,$ a contradiction).
 Thereby $f_1f_2f_3$ is
a blue copy of $\mathcal{C}_3^3.$ Moreover, $\mathcal{P}'=e_{i-3}e_{i-2}$
(we use mod $(n-1)$ arithmetic) is maximal w.r.t.
$W=W_0\cup\{z\}$. Using Lemma \ref{spacial configuration k=3 2} there
is a good $\varpi_{S}$-configuration, say $C$, in $\mathcal{H}_{\rm blue}$ with end vertices in $W$ so that $S\subseteq e_{i-3}\cup e_{i-2}$. Without loss of
generality assume that $u$ is an end vertex of $C$ in $W$. Again, since there is no red copy of $\mathcal{C}_n^3,$
$C\{v,z,v_{2i-1}\}\{v_{2i-1},v_{2i},u\}$ is a blue copy of
$\mathcal{C}_4^3$.

Now let $m\geq 5.$ Clearly $|W_0|=\lfloor\frac{m-1}{2}\rfloor+1\geq 3$. Since there is no red
copy of $\mathcal{C}^3_n$, $\mathcal{P}$ is a maximal path w.r.t.
$W_0$. Now, using Corollary \ref{there is a Pl k=3}, there is a
blue path  of length $\ell'$ between
 $\overline{\mathcal{P}}$, the path obtained from $\mathcal{P}$ by deleting the last $r$ edges, and $W'$ for
 some $r\geq 0$ and $W'\subseteq W_0$ with the properties mentioned in Corollary \ref{there is a Pl
 k=3}. Let $\mathcal{Q}$ be such a blue path so that $\ell'$ is
 maximum. Since $\|\mathcal{P}\|=n-2$, we have $\ell'=2(|W'|-1)=n-2-r.$
Let $x$ and $y$ be the end vertices of $\mathcal{Q}$ in $W'$ and
$T=W_0\setminus W'$. If  $|T|\geq 2,$  then
 $r\geq 3$, a contradiction to Corollary \ref{there is a Pl k=3}. Therefore, we have $|T|\leq 1.$
First let $T=\emptyset$. Clearly $\ell'=2\lfloor\frac{m-1}{2}\rfloor$. If $m$ is even, then $\ell'=m-2.$ Assume that $w$ is a vertex of $e_{i-1}\setminus e_{i-2}$ so that
$w\notin V(\mathcal{Q})$ (the existence of $w$ is guaranteed by
Corollary \ref{there is a Pl k=3}). Since there is no red copy of $\mathcal{C}_n^3,$ the edges $g_1=\{w,z,x\}$ and $g_2=\{y,v_{2i},w\}$ are blue (if the edge $g_j$ for $1\leq j \leq 2$ is red, then $g_jge_{i+1}\ldots e_{n-1}e_1\ldots e_{i-1}$ is a red copy of $\mathcal{C}_n^3,$ a contradiction). Thus $g_1\mathcal{Q}g_2$
is a blue copy of $\mathcal{C}^3_m$. When $m$ is odd, we have  $\ell'=m-1.$ In this case, remove the last two edges of $\mathcal{Q}$ to get a blue path  $\mathcal{Q}'$ of
length $m-3$ so that $v_{2i-2}\notin \mathcal{Q}'.$
 Now we may assume that vertices $x$ and $y'\neq y$
of $W'$ are the end vertices of $\mathcal{Q}'$.
Again, since there is no red copy of $\mathcal{C}^3_n,$  $$\mathcal{Q}'\{y',v_{2i-2},v_{2i}\}\{v_{2i},y,v_{2i-1}\}\{v_{2i-1},z,x\}$$
is a copy of $\mathcal{C}^3_m$ in $\mathcal{H}_{\rm blue}$.\\
Now let $T=\{u\}$. Clearly $l'=2\lfloor\frac{m-1}{2}\rfloor-2$. If
$m$ is odd, then $l'=m-3$ and $r\geq 1$. Thereby,
$$\mathcal{Q}\{y,v_{2i-2},v_{2i}\}\{v_{2i},u,v_{2i-1}\}\{v_{2i-1},z,x\}$$ is
a blue copy of $\mathcal{C}^3_m$.
Now suppose that  $m$ is even. So
$l'=m-4$ and $r\geq 2$. Let $w$ be a vertex of $e_{i-3}\setminus
e_{i-4}$ so that $w\notin V(\mathcal{Q})$. Using Lemma \ref{spacial
configuration k=3 2}, there is a good $\varpi_{S}$-configuration, say $C$, in $\mathcal{H}_{\rm blue}$ with end vertices in $\overline{W}=\{x,y,u,z\}$ so that $S\subseteq (e_{i-2}\setminus \{f_{\mathcal{P},e_{i-2}}\}\cup \{w\})\cup e_{i-1}.$  Since $\mathcal{Q}$ is maximum with the
properties  mentioned in Corollary \ref{there is a Pl k=3},  we can assume that
$y$ and $z$ are end vertices of $C$ in $\overline{W}$. Clearly
$\mathcal{Q}C\{z,u,w'\}\{w',v_{2i},x\}$ is a blue copy of
$\mathcal{C}_m^3,$ where $w'\in e_{i-1}\setminus (V(C)\cup\{f_{\mathcal{P},e_{i-1}}\}).$\\

\noindent \textbf{Case 2. } For every edge
$e_i=\{v_{2i-1},v_{2i},v_{2i+1}\}$, $1\leq i\leq n-1$, and every
vertex $z\in W$  the edges $\{v_{2i-1},v_{2i},z\}$ and
$\{v_{2i},v_{2i+1},z\}$ are blue.\\

\noindent Assume that $W=\{x_1,x_2,\ldots,x_{\lfloor\frac{m-1}{2}\rfloor+2}\}.$
For $1\leq i\leq m-1,$ set
\begin{eqnarray*}
f_i= \left\lbrace
\begin{array}{ll}
\{x_{\frac{i+1}{2}},v_{\frac{3i+1}{2}},v_{\frac{3i+3}{2}}\}  &\mbox{if~}  i~\mbox{is~odd},\vspace{.5 cm}\\
\{v_{\frac{3i}{2}},v_{\frac{3i}{2}+1},x_{\frac{i}{2}+1}\}
&\mbox{if~} i~\mbox{is~even}.
\end{array}
\right.\vspace{.2 cm}
\end{eqnarray*}
Let $\mathcal{Q}=f_1f_2\ldots f_{m-1}$. So for $m$ even,
$\mathcal{Q}\{x_1,v_{\frac{3m}{2}},v_{\frac{3m}{2}+1}\}$ and for
$m$ odd, $\mathcal{Q}\{x_{\frac{m+1}{2}},v_1,v_2\}$ is a blue copy of
$\mathcal{C}^3_m$. }\end{proof}

Before we give the main result, we need the following lemmas.

\begin{lemma}\label{R(Pk3,Pk3)} {\rm \cite{GR}}
For every $k\geq 3$,
\begin{itemize}
\item [\rm{(a)}]~
$R(\mathcal{P}^k_3,\mathcal{P}^k_3)=R(\mathcal{C}^k_3,\mathcal{P}^k_3)=R(\mathcal{C}^k_3,\mathcal{C}^k_3)+1=3k-1$,
\item [\rm{(b)}]~
$R(\mathcal{P}^k_4,\mathcal{P}^k_4)=R(\mathcal{C}^k_4,\mathcal{P}^k_4)=R(\mathcal{C}^k_4,\mathcal{C}^k_4)+1=4k-2$.
\end{itemize}
\end{lemma}

\begin{lemma}\label{R(C3,C4)}
Let $k\geq 3$ be an integer number. Then
$$R(\mathcal{C}^k_4,\mathcal{C}^k_3)= 4k-3.$$
\end{lemma}

\begin{proof}{Suppose indirectly that the edges of $\mathcal{K}^k_{4k-3}$ can be
colored with no red copy of
$\mathcal{C}^k_4$ and no blue copy of $\mathcal{C}^k_3$. By Theorem \ref{R(Pk3,Pk3)}, we must have a blue
copy of $\mathcal{C}^k_4.$ Let $\mathcal{C}=e_1e_2e_3e_4$ be a blue
copy of $\mathcal{C}^k_4$ with edges $e_i=\{v_1,\ldots,v_k\}+(k-1)(i-1)({\rm mod}\  4(k-1)).$ Since there is
no blue copy of $\mathcal{C}^k_3,$ $$(e_2\cup \{v_{4k-4}\}\setminus
\{v_k\})(e_4\cup\{v_k\}\setminus
\{v_1\})(e_1\cup\{v_{3k-3}\}\setminus
\{v_1\})(e_3\cup\{v_1\}\setminus \{v_{3k-2}\})$$ is a red copy
of $\mathcal{C}^k_4.$ This contradiction completes the proof.}\end{proof}
\begin{theorem}\label{main theorem 3uniform}
Conjecture {\rm\ref{cycle}} is true for $k=3$.
\end{theorem}
\noindent\textbf{Proof. }
We give a proof by induction on
$m+n$. The cases $n=m\leq 4$ follow from Lemma \ref{R(Pk3,Pk3)}. By Lemma \ref{R(C3,C4)}, we may assume that $n\geq 5$.
Suppose to the  contrary that the edges of $\mathcal{H}=\mathcal{K}^3_{2n+\lfloor\frac{m-1}{2}\rfloor}$ can be
colored red and blue with no red copy of $\mathcal{C}^3_n$ and
no blue copy of $\mathcal{C}^3_m$. For $n=m$ by the induction hypothesis,
$$R(\mathcal{C}^3_{n-1},\mathcal{C}^3_{n-1})=
2(n-1)+\Big\lfloor\frac{n-2}{2}\Big\rfloor<
2n+\Big\lfloor\frac{n-1}{2}\Big\rfloor.$$ So we may assume that
there is a red copy of $\mathcal{C}^3_{n-1}$ in $\mathcal{H}$.
Using Lemma \ref{cn-1 implies cm k=3} we have a blue copy of
$\mathcal{C}^3_{n}$; a contradiction. For $n> m$, we have $n-1\geq m$ and since
$$R(\mathcal{C}^3_{n-1},\mathcal{C}^3_{m})=
2(n-1)+\Big\lfloor\frac{m-1}{2}\Big\rfloor<
2n+\Big\lfloor\frac{m-1}{2}\Big\rfloor,$$ we may assume that
$\mathcal{C}^3_{n-1}\subseteq \mathcal{H}_{\rm red}$.
Using  Lemma \ref{cn-1 implies cm k=3} we have a blue copy of
$\mathcal{C}^3_{m}$; a contradiction.
$\hfill\blacksquare$

Using Theorems \ref{equivalent1} and \ref{main theorem 3uniform}    we have the following.
\begin{theorem}\label{main theorem 3uniform2}
Conjecture {\rm\ref{path-cycle}} is true for $k=3$.
\end{theorem}

\section{An equivalent variation of Conjecture \ref{path-cycle} for $k\geq 4$}

In this section, we show that for fixed $m\geq 3$ and  $k\geq 4$  Conjecture \ref{path-cycle} is equivalent to (only) the last equality for any  $2m\geq n\geq m\geq 3$. Indeed,  we demonstrate that
for  fixed $m\geq 3$ and $k\geq 4,$ Conjecture
\ref{cycle} holds for every  $n\geq m$ if and only if it holds for every $2m\geq n\geq m.$ Then, using Theorem \ref{equivalent1}, the desired result is achieved.
We will use the following lemma later as follows: We will exhibit
two disjoint red copies of $\mathcal{C}^k_{\lfloor n_0/2\rfloor}$ and $\mathcal{C}^k_{\lceil n_0/2\rceil}$ and this will be enough to ensure the existence of a red copy of $\mathcal{C}^k_{n_0}$ or a blue copy of $\mathcal{C}^k_{m}$ for $n_0\geq 2m+1.$

\begin{lemma}\label{CnCmimplies Cn+m}
Let $n,m,t$ and $k$ be positive integer numbers so that $n\geq
m\geq 3,$ $k\geq 4$ and $t\geq (k-1)(n+m).$ Let
$\mathcal{H}=\mathcal{K}_t^k$ be $2$-edge colored red and blue and
$\mathcal{C}_1=\mathcal{C}^k_n$ and
$\mathcal{C}_2=\mathcal{C}^k_m$ be two disjoint cycles in
$\mathcal{H}_{\rm red}$. There is a red copy of
$\mathcal{C}^k_{n+m}$ or  a blue copy of  $\mathcal{C}^k_{\ell}$ for every
$3\leq \ell\leq m$.

\end{lemma}

\begin{proof}{Let $\mathcal{C}_1=e_1e_2\ldots e_{n}$ and $\mathcal{C}_2=f_1f_2\ldots f_{m}$ be copies  of $\mathcal{C}_{n}^k$ and $\mathcal{C}_{m}^k$ in $\mathcal{H}_{\rm
red}$ with edges $$e_i=\{v_1,v_2,\ldots,v_k\}+(i-1)(k-1) ({\rm
mod}\ n(k-1))$$  and
$$f_i=\{u_1,u_2,\ldots,u_k\}+(i-1)(k-1) ({\rm mod}\  m(k-1)).$$
 Without loss of generality assume
that there is no copy of $\mathcal{C}^k_{n+m}$ in
$\mathcal{H}_{\rm red}$.
%
%
%
Let $$g_1=(e_1\setminus \{v_{k-1},v_k\})\cup \{u_{k-1},u_k\}$$
and
$$h_1=(f_1\setminus \{u_{k-1},u_k\})\cup \{v_{k-1},v_k\}.$$ At least
one of the edges $g_1$ or $h_1$, say $s_1$, is blue (otherwise,
$h_1e_2e_3\ldots e_{n}g_1f_2f_3\ldots f_{m}$ is a red copy of
$\mathcal{C}^k_{n+m}$, a contradiction to our assumption). Assume
that $x_2\in s_1\cap e_1$ and $y_2\in s_1\cap
f_1$ are vertices with maximum indices. Now set
$\mathcal{P}_1=s_1$, $$g_2=(e_2\setminus
\{v_k,v_{2k-2},v_{2k-1}\})\cup \{x_2,u_{2k-2},u_{2k-1}\}$$ and
$$h_2=(f_2\setminus \{u_k,u_{2k-2},u_{2k-1}\})\cup
\{y_2,v_{2k-2},v_{2k-1}\}.$$
Similarly, at least one of the edges $g_2$ or $h_2$, say $s_2$, is blue (otherwise, $h_2e_3e_4\ldots e_{n} e_1 g_2 f_3 f_4$ $ \ldots f_{m} f_1$ is a red copy of
$\mathcal{C}^k_{n+m}$, a contradiction).
Set $\mathcal{P}_2=\mathcal{P}_1s_2$ and continue this process
$\ell-2$ times. Note that, with this approach, in the $i$-th step,
$2\leq i\leq \ell-2$, we have $$g_i=(e_i\setminus \{v_{(i-1)(k-1)+1},
v_{i(k-1)}, v_{i(k-1)+1}\})\cup $$$$\{x_i,
u_{i(k-1)}, u_{i(k-1)+1}\}$$ and $$h_i= (f_i\setminus
\{u_{(i-1)(k-1)+1}, u_{i(k-1)}, u_{i(k-1)+1}\}) \cup
$$$$\{y_i, v_{i(k-1)}, v_{i(k-1)+1}\},$$ where $x_i\in
s_{i-1}\cap e_{i-1}$ and $y_i\in
s_{i-1} \cap f_{i-1}$ are vertices
with maximum indices. Since there is no red copy of $\mathcal{C}^k_{n+m}$,  at least one of the edges $g_i$ or $h_i$, say
$s_i$, is blue. So
 clearly $\mathcal{P}_i=\mathcal{P}_{i-1}s_i$ is a blue path of length
 $i$. Now let $$g_{\ell-1}=\big(e_{n}\setminus
 \{v_{(n-1)(k-1)+1},v_{(n-1)(k-1)+2},v_{1}\}\big)\cup $$ $$
 \{x_{1},u_{(\ell-1)(k-1)},u_{(\ell-1)(k-1)+1}\}$$ and  $$h_{\ell-1}=\big(f_{\ell-1}\setminus
\{u_{(\ell-2)(k-1)+1}, u_{(\ell-1)(k-1)},
u_{(\ell-1)(k-1)+1}\}\big)\cup $$ $$ \{y_{\ell-1}, v_{(n-1)(k-1)+1},
v_{(n-1)(k-1)+2}\},$$ so that  $x_{1}\in s_1 \cap e_1$ is a vertex with minimum index and  $y_{\ell-1}\in
s_{\ell-2}\cap f_{\ell-2}$ is a vertex
with maximum index. Again, at least one of the edges $g_{\ell-1}$ or $h_{\ell-1}$ is blue (otherwise, $g_{\ell-1}e_1e_2\ldots e_{n-1} h_{\ell-1}f_{\ell-2}\ldots f_mf_1\ldots f_{l}$ is a red copy of
$\mathcal{C}^k_{n+m}$, a contradiction).
If the edge  $g_{\ell-1}$ is blue, then set
$$g_{\ell}=(e_{\ell-1}\setminus \{v_{(\ell-2)(k-1)+1}, v_{(\ell-1)(k-1)},
v_{(\ell-1)(k-1)+1}\})\cup$$$$ \{x_{\ell-1}, u_{(\ell-2)(k-1)+k-2},
u_{(\ell-1)(k-1)+1}\}$$ and
$$h_{\ell}=(f_{\ell-1}\setminus \{u_{(\ell-2)(k-1)+1}, u_{(\ell-2)(k-1)+k-2}, u_{(\ell-1)(k-1)+1}\})\cup$$$$ \{y_{\ell-1}, v_{(\ell-1)(k-1)}, v_{(\ell-1)(k-1)+1}\},$$ where $x_{\ell-1}\in s_{\ell-2}\cap e_{\ell-2}$ and $y_{\ell-1}\in s_{\ell-2}\cap f_{\ell-2}$ are vertices with maximum indices. Clearly,  at least one of the edges $g_{\ell}$ or $h_{\ell}$ is blue (otherwise, $h_{\ell}e_{\ell}e_{\ell+1}\ldots e_ne_1\ldots e_{\ell-2}g_{\ell}f_{\ell}f_{\ell+1}\ldots f_nf_1$ $\ldots f_{\ell-2}$ is a red copy of $\mathcal{C}^k_{n+m}$). If the edge $g_{\ell}$ is blue, then $\mathcal{P}_{\ell-2}g_{\ell}g_{\ell-1}$ is
a blue copy of $\mathcal{C}_{\ell}^k$. Otherwise, $\mathcal{P}_{\ell-2}h_{\ell}g_{\ell-1}$   is a blue copy of $\mathcal{C}_{\ell}^k$. Now we may assume that the edge $g_{\ell-1}$ is red and so the edge $h_{\ell-1}$ is blue. In this case
set $$g'_{\ell}=(e_n\setminus
\{v_{(n-1)(k-1)+2}, v_{1}\})\cup \{u_{m(k-1)},y_1\}$$ and
$$h'_{\ell}=(f_m\setminus \{u_{m(k-1)}, u_1\})\cup
\{v_{(n-1)(k-1)+2},x_1\},$$ where $x_1\in s_1\cap e_1$ and $y_1 \in s_1\cap f_1$ are vertices
with minimum indices. Since at least one of the edges  $g'_{\ell}$ and $h'_{\ell}$, say
$s_{\ell}$, is blue, then $\mathcal{P}_{\ell-2}h_{\ell-1}s_{\ell}$ is a blue copy of
$\mathcal{C}_{\ell}^k$, which completes the proof. }\end{proof}

\begin{theorem}\label{Am}
Let $m\geq 3$ be a fixed integer, $k\geq 4$ and $A_m=\{n\geq m :
R(\mathcal{C}_{n}^k,\mathcal{C}_{m}^k)=(k-1)n+\lfloor\frac{m-1}{2}\rfloor$\}.
 If $[m,2m]\subseteq A_m,$ then $A_m=[m,\infty]$.
\end{theorem}

\begin{proof}{ Suppose to the contrary that $A_m\neq[m,\infty]$. Let $n_0$ be the minimum element of
$[m,\infty]-A_m$. Note that $n_0\geq 2m+1$. Let
$\mathcal{H}=\mathcal{K}_{(k-1)n_0+\lfloor\frac{m-1}{2}\rfloor}^k$
be $2$-edge colored red and blue with no  red copy of
$\mathcal{C}_{n_0}^k$ and no  blue copy of $\mathcal{C}_{m}^k$. We
consider the following cases:\\

\noindent\textbf{Case 1. } $n_0$ is even.\\
Clearly $\frac{n_0}{2}> m$. Since $n_0$ is the minimum integer
in $[m,\infty]-A_m,$ then
$$R(\mathcal{C}_{\frac{n_0}{2}}^k,\mathcal{C}_{m}^k)=(k-1)\frac{n_0}{2}+\lfloor\frac{m-1}{2}\rfloor<(k-1)n_0+\lfloor\frac{m-1}{2}\rfloor.$$
Therefore, there is a copy of $\mathcal{C}_{\frac{n_0}{2}}^k$ in
$\mathcal{H}_{\rm red}$, say $\mathcal{C}_1$.  Remove $\mathcal{C}_1$ from $\mathcal{H}$ to get a hypergraph $\mathcal{H}'.$
 Since the reminder hypergraph $\mathcal{H}'$ has equal to $R(\mathcal{C}_{\frac{n_0}{2}}^k,\mathcal{C}_{m}^k)$ vertices,
 there is red copy of
$\mathcal{C}_{\frac{n_0}{2}}^k$, say
$\mathcal{C}_2$,  disjoint from $\mathcal{C}_1$. Using Lemma
\ref{CnCmimplies Cn+m} we have  a red copy of $\mathcal{C}_{n_0}^k$ or a blue  copy of $\mathcal{C}_{m}^k$,  a contradiction to our assumption.\\

\noindent\textbf{Case 2. } $n_0$ is odd.\\
By an argument similar  to case 1, we may assume that there are
disjoint copies of $\mathcal{C}^k_{\frac{n_0-1}{2}}$ and
$\mathcal{C}^k_{\frac{n_0+1}{2}}$ in $\mathcal{H}_{\rm red}$. So
using Lemma \ref{CnCmimplies Cn+m} we have   a red copy of $\mathcal{C}_{n_0}^k$ or a blue copy of
$\mathcal{C}_{m}^k$. This contradiction completes the proof.
}\end{proof}

\begin{remark}\label{equi}
Note that Theorem {\rm\ref{Am}} shows that for  fixed $m\geq 3$ and $k\geq 4,$  Conjecture {\rm\ref{cycle}} holds for each  $n\geq m$  if and only if it holds for each  $ 2m\geq n \geq m.$ This show, using Theorem {\rm\ref{equivalent1}}, that for fixed $m\geq 3$ and  $k\geq 4,$  Conjecture {\rm\ref{path-cycle}} is equivalent to (only) the last equality for any  $2m\geq n\geq m\geq 3$.
We will use this result in the next section,  to demonstrate  that Conjecture {\rm\ref{path-cycle}}
 is true for $m=3.$ \end{remark}

%
%
%
%

\section{Cycle-triangle Ramsey number in uniform hypergraphs}

In this section, we show that Conjecture \ref{path-cycle}
holds for any $n\geq m=3.$ By Remark \ref{equi}, it only suffices to prove that Conjecture \ref{cycle} is true for $m=3$ and $3\leq n\leq 6.$
First, we establish some
essential lemmas.


\begin{lemma}\label{C4k implies C5k}
Let $\mathcal{H}=\mathcal{K}^k_{i(k-1)+1}$, $i=5,6$, be $2$-edge
colored red and blue with no blue copy of $\mathcal{C}^k_3.$ If
there is a blue copy of $\mathcal{C}^k_4,$ then there is a copy of
$\mathcal{C}^k_i$ in $\mathcal{H}_{\rm red}$.
\end{lemma}

\begin{proof}{
Suppose that the edges of $\mathcal{H}$ are 2-colored with no blue
copy of $\mathcal{C}^k_3$.  Let $\mathcal{C}=e_1e_2e_3e_4$ be
 a blue copy of $\mathcal{C}^k_4$ with edges $$e_i=\{v_1,\ldots,v_k\}+(i-1)(k-1)({\rm mod}\  4(k-1)),\hspace{0.5 cm} 1\leq i \leq 4,$$  and $W=V(\mathcal{H})\setminus V(\mathcal{C})$.
For $i=5$, since there is no blue copy of $\mathcal{C}^k_3,$
$$(e_2\cup \{v_{4k-4}\}\setminus
\{v_k\})(e_4\cup\{v_{k}\}\setminus
\{v_1\})(e_1\cup\{v_{3k-3}\}\setminus \{v_1\})$$$$
(\{v_{3k-3},v_1\}\cup W\setminus \{w_1,w_2\}) (e_3\cup
\{v_1,w_1\}\setminus \{v_{3k-3},v_{3k-2}\}),$$ where $w_1,w_2\in
W,$ is a red copy of $\mathcal{C}^k_5.$ For $i=6$, partition the
vertices of $W$ into three sets  $A$, $B$ and $C$  with $|A|=|B|=k-2$ and $|C|=3$.
Let $C=\{u_1,u_2,u_3\}$.
 Since there is
no blue copy of $\mathcal{C}^k_3,$ $$(e_2\cup
\{u_1,v_{4k-4}\}\setminus
\{v_k,v_{k+1}\})(e_4\cup\{v_{k+1}\}\setminus \{v_1\}) (e_3\cup
\{v_{k},u_2\}\setminus \{v_{2k-1},v_{2k}\})$$$$ (A\cup
\{v_k,v_{2k}\}) (e_1\cup\{v_{2k}\}\setminus
\{v_k\})(B\cup\{v_1,v_{2k-1}\})$$ is a red copy of
$\mathcal{C}^k_6.$
 }\end{proof}

Also, we need the following  remark of \cite[Remark 3]{GR}. For the sake of completeness, we also represent it's proof here.
\begin{remark}\label{r1} If $\mathcal{K}^k_N$ with $N\geq k+1$ is $2$-edge colored
and both colors are used at least once, then there are two edges
of distinct colors intersecting in $k-1$ vertices.
\end{remark}
\begin{proof}{
Select a red edge $e$ and a blue edge $f$ with maximum
intersection and suppose that $m=|e\cap f|<k-1.$ Let $g$ be an
edge that contains $e\cap f$ and intersects $e\setminus f$ in
$\lfloor\frac{k-m}{2}\rfloor$ vertices and intersects $f\setminus
e$ in $\lceil\frac{k-m}{2}\rceil$ vertices. Now, either $g,e$ or
$g,f$ are two edges of distinct colors intersecting in more than
$m$ vertices, this contradicts the choice of $e$ and $f$.
}\end{proof}

The following lemma is a modified version of Lemma $7$ in
\cite{GR}. But, for the sake of completeness, we state a proof
here.


\begin{lemma}\label{two disjoint edges}
Let $k\geq 3$ and $t\geq 5$. If $\mathcal{H}=\mathcal{K}^k_{t(k-1)+1}$ is $2$-edge
colored red and blue with no red copy of $\mathcal{C}^k_t$ and no
blue copy of $\mathcal{C}^k_3$, then there exist two intersecting
pairs of red-blue edges $e_1,e_2$ and $f_1,f_2$ so that $(e_1\cup
e_2)\cap(f_1\cup f_2)=\emptyset.$
\end{lemma}

\begin{proof}{Note that both colors are used at least once, for otherwise there is either a red copy of $\mathcal{C}^k_t$ or a blue copy of $\mathcal{C}^k_3$, a contradiction.
As noted in Remark \ref{r1},
we can select $e_1=\{v_1,v_2,\ldots,v_{k}\}\in \mathcal{H}_{\rm
red}$ and $e_2=\{v_2,v_3,\ldots,v_{k+1}\}\in \mathcal{H}_{\rm
blue}.$  Let $W=V(\mathcal{H})\setminus
\{v_1,v_2,\ldots,v_{k+1}\}.$
 If $W$ is not monochromatic, we can find favorable edges $f_1$ and
 $f_2$ in $W$ based on Remark \ref{r1}. Thus all edges of $W$ are red (otherwise, there is a blue copy of $\mathcal{C}^k_3,$ a
 contradiction). Let $W_1$ and $W_2$ be disjoint subsets of $W$ so
 that $|W_1|=|W_2|=k-1.$ We
may assume that $g_1=\{v_{1}\}\cup W_1$ is red. If it is not,
let $g_2=\{v_{k+1}\}\cup W_2.$ If $g_2$ is red, then $g_2,e_2$ and
$W_1\cup \{w\},g_1,$ where $w\in W\setminus (W_1\cup W_2),$ are
favorable red-blue pairs. If $g_2$ is blue, then $e_1,g_1$ and
$W_2\cup\{w\},g_2,$ where $w\in W\setminus (W_1\cup W_2),$ are
desired. Indeed, for every $W'\subseteq W$ with $|W'|=k-1,$ the edge
$\{v_{1}\}\cup W'$ is red. Note that the edge
$f_1=\{v_k\}\cup W_2$ is blue. Otherwise, the edges $f_1, e_1, g_1$
with the $t-3$ suitable edges with vertices in $W$ give a red copy of
$\mathcal{C}^k_t.$ Actually, using a similar argument, for every
$i,$ $2\leq i\leq k,$ and every $W'\subseteq W$ with $|W'|=k-1,$
the edge $\{v_{k+2-i}\}\cup W'$ is blue. Therefore, $g_1, \{v_{k-1}\}\cup
W_1$ and $\{w\}\cup W_2,f_1,$ where $w\in W\setminus (W_1\cup
W_2),$ are two disjoint red-blue pairs that satisfy the mentioned
condition.
 }\end{proof}
\begin{corollary}\label{two disjoint intersecting edges}
Let $k\geq 3$, $t\geq 5$ and $\mathcal{H}=\mathcal{K}^k_{t(k-1)+1}$ be $2$-edge
colored red and blue. If there is no red copy of $\mathcal{C}^k_t$ and no blue copy of $\mathcal{C}^k_3$, then there exist two
disjoint intersecting pairs of red-blue edges $e_1,e_2$ and
$f_1,f_2$ so that $|e_1\cap e_2|=|f_1\cap f_2|=k-1.$
\end{corollary}

\begin{proof}{
Using Lemma \ref{two disjoint edges}, we can find two intersecting
pairs of red-blue edges $e_1,e_2$ and $f_1,f_2$ so that $(e_1\cup
e_2)\cap(f_1\cup f_2)=\emptyset.$ Now, consider the complete
hypergraphs on $V(e_1)\cup V(e_2),$ say $\mathcal{G}_1,$ and
$V(f_1)\cup V(f_2),$ say $\mathcal{G}_2$. Since $\mathcal{G}_1$
and $\mathcal{G}_2$ have edges of both colors, applying Remark \ref{r1} we can  find two edges of distinct colors intersecting in
$k-1$ vertices in each of $\mathcal{G}_1$ and $\mathcal{G}_2.$}\end{proof}

\begin{lemma}\label{R(C^k_3,Ck_5)}
Let $k\geq 4$ be an integer number. Then
$$R(\mathcal{C}^k_5,\mathcal{C}^k_3)= 5k-4.$$
\end{lemma}

\begin{proof}{
Suppose to the contrary that the edges of
$\mathcal{H}=\mathcal{K}^k_{5k-4}$ can be colored red and blue
with no   red copy of $\mathcal{C}^k_5$ and no  blue copy of
$\mathcal{C}^k_3$. Apply Corollary \ref{two disjoint intersecting
edges} for $t=5$ to find $$e_1=\{v_1,v_2,\ldots,v_{k}\}\in
\mathcal{H}_{\rm red}, e_2=\{v_2,v_3,\ldots,v_{k+1}\}\in
\mathcal{H}_{\rm blue}$$ and $$f_1=\{w_1,w_2,\ldots,w_{k}\}\in
\mathcal{H}_{\rm red}, f_2=\{w_2,w_3,\ldots,w_{k+1}\}\in
\mathcal{H}_{\rm blue}$$ so that  $$|e_1\cap e_2|=|f_1\cap
f_2|=k-1, (e_1\cup e_2)\cap(f_1\cup f_2)=\emptyset.$$ Set
$$W=V(\mathcal{H})\setminus \{v_1,\ldots, v_{k+1}, w_1, \ldots,
w_{k+1}\}$$ and partition the vertices of $W$ into three sets
$A,B$ and $C$ with $|A|=|B|=|C|=k-2$ (note that this is possible
since $|W|=3k-6$). Without loss of generality assume that $T$ and
$S$ are two subsets of $W$ so that $T\subseteq A,$ $S\subseteq B$
and $|T|+|S|=k-2.$ Set $T'=A\setminus T,$ $S'=B\setminus S,$ $t\in
T$ and $s\in S'$. The proofs of the following claims are similar,
so we only give a proof for Claim \ref{c1}.

\begin{claim}\label{c1}
If there
is $i\in \{2,3,\ldots,k\}$ so that the edge $e=\{v_{k+1},w_{k+2-i}\}\cup T
\cup S$ is blue, then for every $j\in\{2,3,\ldots,k\},$
 the edge
$f=\{w_{k+1},v_{k+2-j}\}\cup (T\setminus \{t\})\cup (S\cup \{s\}),$  is
also blue.
\end{claim}
\begin{proof}{ For every $\ell\in
\{2,3,\ldots,k\}\setminus \{i\}$ (resp. $\ell\in
\{2,3,\ldots,k\}\setminus \{j\}$), the edge $h_1=(T'\cup\{t\})\cup
S'\cup\{w_{k+2-\ell}\}$ (resp. $h_2=C\cup\{v_{k+2-\ell},w_{k+2-i}\}$) is red. Otherwise
$f_2h_1e$ (resp. $e_2eh_2$) is a blue copy of $\mathcal{C}^k_3$, a
contradiction to our assumptions. Now, if the edge  $f$ is red, then
$f_1h_1fe_1h_2$ is a red copy of
$\mathcal{C}^k_5$. This contradiction finishes the proof.}\end{proof}

\begin{claim}\label{c2}
If there is $i\in \{2,3,\ldots,k\}$  so that the edge $\{w_{k+1},v_{k+2-i}\}\cup T \cup
S$ is blue, then for every $j\in\{2,3,\ldots,k\},$ the edge
$\{v_{k+1},w_{k+2-j}\}\cup (T\setminus \{t\})\cup (S\cup \{s\})$ is also
blue.
\end{claim}

\begin{claim}\label{c3}
If there are $i,j\in \{2,3,\ldots,k\}$  so that $\{v_{k+2-i},w_{k+2-j}\}\cup T \cup
S$ is blue, then for every $\ell\in\{2,3,\ldots,k\}\setminus \{i\}$
and $\ell'\in\{2,3,\ldots,k\}\setminus \{j\}$, the edge
$\{v_{k+2-\ell},w_{k+2-\ell'}\}\cup (T\setminus \{t\})\cup (S\cup \{s\})$ is also
blue.
\end{claim}

One can easily check that  the edge $g_1=A\cup \{v_{k+1},w_{k}\}$ is
red. Otherwise, apply Claims \ref{c1} and \ref{c2} alternatively $k-3$ times to
find a blue copy of $\mathcal{C}^k_3$ if $k$ is even and apply Claims \ref{c1} and
\ref{c2} alternatively $k-2$ times to find a blue copy of $\mathcal{C}^k_4$ if $k$
is odd. In the first case, $g_1e_2(B\cup\{w_{k+1},v_{k},a\}\setminus
\{b\}),$ where $a\in A$ and $b\in B,$ is a blue copy of $\mathcal{C}^k_3$
that is a contradiction to our assumptions. In the second case,
$g_1f_2(B\cup\{w_{k+1},v_{k}\})e_2$ is a blue copy of $\mathcal{C}^k_4.$ Now,
using Lemma \ref{C4k implies C5k} we can find a red copy of
$\mathcal{C}^k_5,$ a contradiction. Also, a similar discussion can
be used to show that the edges $g_2=C\cup \{w_{k+1},v_{k-1},u\}\setminus
\{u'\}$, where $u\in A$ and $u'\in C$, and $g_3=B\cup \{v_{k},w_{k-1}\}$  are
red. Thereby, $g_1f_1g_3e_1g_2$ is a copy of $\mathcal{C}^k_5$ in
$\mathcal{H}_{\rm red}.$ This contradiction completes the proof.}\end{proof}

The proof of the following statement is similar to the proof of
Lemma \ref{R(C^k_3,Ck_5)}. So we only present the outline of the
proof.

\begin{lemma}\label{R(C3,C6)}
Let $k\geq 4$ be an integer number. Then
$$R(\mathcal{C}^k_6,\mathcal{C}^k_3)= 6k-5.$$
\end{lemma}

\begin{proof}{
Suppose that the edges of $\mathcal{H}=\mathcal{K}^k_{6k-5}$ are
colored red and blue with no  red copy of $\mathcal{C}^k_6$ and no
blue copy of $\mathcal{C}^k_3$. Using Corollary \ref{two disjoint
intersecting edges}, select edges $$e_1=\{v_1,v_2,\ldots,v_{k}\}\in
\mathcal{H}_{\rm red}, e_2=\{v_2,v_3,\ldots,v_{k+1}\}\in
\mathcal{H}_{\rm blue}$$ and $$f_1=\{w_1,w_2,\ldots,w_{k}\}\in
\mathcal{H}_{\rm red}, f_2=\{w_2,w_3,\ldots,w_{k+1}\}\in
\mathcal{H}_{\rm blue}.$$ Set $$W=V(\mathcal{H})\setminus
\{v_1,\ldots,v_{k+1},w_1, \ldots, w_{k+1}\}$$ and for a vertex $u\in W$ partition the
vertices of $W\setminus
\{u\}$ into four sets $A,B,C$ and $D$ of size
$k-2$. Suppose that $T$ and $S$ are two subsets of
$W$ so that $T\subseteq A,$ $S\subseteq B$ and $|T|+|S|=k-2.$ Set
$T'=A\setminus T,$ $S'=B\setminus S,$ $t\in T$ and $s\in S'$.
Arguments similar to the proofs of Claims \ref{c1} and \ref{c2} in Lemma \ref{R(C^k_3,Ck_5)}, yield the same claims here (note that, for proofs, more
details are required). Consequently, we may assume that the edges
$g_1=C\cup \{v_{k+1},w_{k}\}$ and $g_2=D \cup \{v_{k-1},w_{k+1},u'\}\setminus
\{u''\},$ where $u'\in C$ and $u''\in D,$ are red.
 Also one can easily see that the following claim
holds.

\begin{claim}
If there is $i\in \{2,3,\ldots,k\}$ so that the edge $f=\{v_{k+2-i},u\}\cup T\cup S$ is blue,
 then for every $j\in\{2,3,\ldots,k\}\setminus \{i\},$ the edge
$f'=\{w_{k+1},v_{k+2-j}\}\cup (T\setminus \{t\})\cup (S\cup \{s\})$  is
also blue.
\end{claim}

\begin{proof}{ Using the above arguments we may assume that for every $2\leq \ell\leq k$, the edge $h_1=\{v_{k+1},w_{k+2-\ell}\}\cup C$ is red. Also, for any $\ell'$ and $\ell''\neq \ell$ with $2\leq \ell', \ell''\leq k$ and $\ell'\neq i,j$, the edge $h_2=\{v_{k+2-\ell'},t,w_{k+2-\ell''}\}\cup (D\setminus\{d\})$ (resp. $h_3=\{v_{k+1},u\}\cup T'\cup S'$) is red, where $d\in D$. Otherwise $e_2fh_2$ (resp. $e_2fh_3$) is a blue copy of $\mathcal{C}^k_3$, a contradiction. Now if $f'$ is red, then $h_3f'e_1h_2f_1h_1$ is a red copy of $\mathcal{C}^k_6$, a contradiction.}\end{proof}

Similarly, we conclude that the edges $g_3=A \cup \{u,w_{k-1}\}$ and
$g_4=B\cup \{u,v_{k}\}$ are red.
 Therefore, $g_1
f_1 g_3 g_4 e_1 g_2$
 is a red copy of
$\mathcal{C}^k_6$. This contradiction finishes the proof.}\end{proof}

The following theorem is an immediate consequence of Theorems \ref{main theorem 3uniform} and  \ref{Am},  and  Lemmas \ref{R(Pk3,Pk3)}, \ref{R(C3,C4)}, \ref{R(C^k_3,Ck_5)} and
\ref{R(C3,C6)} (see Remark \ref{equi}).

\begin{theorem}\label{m3thm}
Conjecture {\rm\ref{cycle}} holds for $m=3$.
\end{theorem}

Using Theorems \ref{equivalent1} and \ref{m3thm} we have:

\begin{theorem}
Conjecture {\rm\ref{path-cycle}} is true for $m=3$.
\end{theorem}

\end{document}